\documentclass[12pt]{amsart}

\usepackage{amscd,amsmath,amssymb,amsfonts}
\usepackage{graphicx, psfrag}
\usepackage[mathscr]{eucal}
\usepackage{color}

\newtheorem{thm}{Theorem}[section]
\newtheorem{lemma}[thm]{Lemma}

\newtheorem{prop}[thm]{Proposition}

\newcommand{\ps}{\psfrag}

\newtheorem{defn}[thm]{Definition}

\newcommand{\R}{{\mathbb R}}

\newcommand{\E}[1]{\mathsf E(#1)}
\newcommand{\SH}[2]{\mathsf{SH}_{#1}(#2)}

\renewcommand{\o}{\mathsf}	 

\newcommand{\HP}{\mathbf{H}^{2}}

\DeclareMathOperator{\arccosh}{arccosh}

\newcommand{\HD}{d}
\newcommand{\HDS}{{\HD}^2}
\newcommand{\BHD}{\rho}

\newcommand{\B}{\mathbf{B}}

\newcommand{\vu}{\mathsf{u}}
\newcommand{\vq}{\mathsf{q}}
\newcommand{\vp}{\mathsf{p}}
\newcommand{\vy}{\mathsf{y}}
\newcommand{\vx}{\mathsf{x}}
\newcommand{\vv}{\mathsf{v}}

\newcommand{\da}{\textsc{a}}
\newcommand{\db}{\textsc{b}}
\newcommand{\dc}{\textsc{c}}
\newcommand{\dd}{\textsc{d}}
\newcommand{\dk}{\textsc{k}}

\begin{document}

\title[Bisectors in the bidisk]{Bisectors determining unique pairs of points in the bidisk}
\author[Charette]{Virginie Charette}
  \address{D\'epartement de math\'ematiques\\ Universit\'e de Sherbrooke\\
  Sherbrooke, Quebec, Canada}
  \email{v.charette@usherbrooke.ca}

\author[Drumm]{Todd A.\ Drumm}
\address{Department of Mathematics\\ Howard University\\
     Washington, DC 20059  USA }
    \email{tdrumm@howard.edu}

\author[Kim]{Youngju Kim$\dag$}
  \address{Department of Mathematics Education\\Konkuk University\\ Neungdong-ro 120 Gwangjin-gu, Seoul 05029, Republic of  Korea}
   \email{geometer2@konkuk.ac.kr}

\date{\today}

\thanks{Charette gratefully acknowledges partial support from the
  Natural Sciences and Engineering Research Council of Canada, and thanks Konkuk University for its hospitality while this paper was finished.\\
 Drumm gratefully acknowledges partial support from NSF grant DMS1107367
``Research Networks in the Mathematical Sciences: Geometric structures And Representation varieties" (the GEAR Network).\\
 $\dag$ This research was supported by Basic Science Research Program through the National Research Foundation of Korea (NRF) funded by the Ministry of Education, Science and Technology (NRF-2011-0021240).}

\begin{abstract}
Bisectors are equidistant hypersurfaces between two points and are basic
objects in a metric geometry. They play an important part in understanding the action
of subgroups of isometries  on a metric space. In many
metric geometries (spherical, Euclidean, hyperbolic, complex hyperbolic, to name a few) bisectors do not uniquely determine a pair of points, in the following sense\,:
completely different sets of points share a common bisector. The above examples of this
non-uniqueness are all rank $1$  symmetric spaces.
 However, as we show in this paper, bisectors in the usual $L^2$ metric are such for a unique pair of points
in the rank $2$ geometry
$\HP \times\HP$.

\end{abstract}

\maketitle

\section{Introduction}
Suppose $X$ is a metric space with metric $\BHD$ and isometry group $\o{Isom}(X)$.
If $\Gamma\subset\o{Isom}(X)$ is a
discrete group that acts properly on $X$ then\,:
\[
\Delta_{\Gamma} (x) =  \{ y\in X \, \vert \, \BHD(y, x) \leq \BHD (\gamma (y),x ) \, \forall \gamma\in\Gamma \}
\]
is called the \emph{Dirichlet domain} of the action of $\Gamma$ centered at the point $x$. It consists of all
the points closer to $x\in X$ than any other element of the $\Gamma$-orbit of $x$. The boundary
of the Dirichlet domain consists of pieces of
\emph{bisectors}, which are equidistant hypersurfaces between two points\,:
\[
\E{x,y} =  \{ z\in X \, \vert \, \BHD(x,z) = \BHD (y,z)\}.
\]

Bisectors are basic and well known  objects in  many metric homogeneous geometries.
\begin{itemize}
\item In the Euclidean spaces $\mathbf{R}^n$, bisectors are
hyperplanes.
\item In  the spherical spaces $S^n$, bisectors are the higher dimensional analog of great circles.
\item In the upper half-space models of real hyperbolic spaces $\mathbf{H}^n$, bisectors are hemispheres centered on the boundary
 or vertical hyperplanes orthogonal to the boundary.
\item In the  complex hyperbolic spaces $\mathbf{CH}^n$, bisectors are well-studied and used in understanding
discrete group actions. (See for instance \cite{Aebischer-Miner, Goldman-Parker:triangle, Mostow}.)
\end{itemize}

All of these are rank $1$ geometries. 
In contrast, the {\em bidisk} $\HP\times\HP$ is a rank $2$
geometry.  (Note that we are working with the usual $L^2$ metric.)  Here, bisectors come in two shapes, depending on whether $\vx,\vy\in\mathbf{H}^2 \times\mathbf{H}^2$ share a common coordinate\,:
\begin{itemize}
\item If $\vx$ and $\vy$ share a common coordinate, then $\E{\vx,\vy}$ is foliated by copies of the hyperbolic plane $\HP$ and contains a 1-dimensional family of flats;
\item otherwise, $\E{\vx,\vy}$ contains no copy of the hyperbolic plane and a single flat~\cite{CDD}.  We will say that this bisector is {\em generic}.
\end{itemize}

The aim of this paper is to prove that in the bidisk, in the generic case, pairs of points are uniquely determined by their bisectors.
\begin{thm}\label{thm:main}
Let $\vx,\vy,\vu,\vv\in\mathbf{H}^2 \times\mathbf{H}^2$ such that
$\E{\vx,\vy}=\E{\vu,\vv }$.  Assume that the bisector is generic.  Then\,:
\begin{equation*}
\{ \vx,\vy \} =\{\vu, \vv \}.
\end{equation*}
\end{thm}

One can easily see that uniqueness fails in the non-generic case, which behaves like the hyperbolic plane.

The paper is structured as follows.  In \S\ref{sec:prelim}, we introduce basic features of the bidisk and the hyperbolic plane, in particular {\em square hyperbolae}, whose products are the leaves in a foliation of a bisector.  In \S\ref{sec:leaves}, we show that if two bisectors are equal, then they must share the same leaves up to re-parametrization.  In \S\ref{sec:collinear}, we next show a collinearity condition on the points defining the bisectors.  Finally, \S\ref{sec:SH} completes the proof of Theorem~\ref{thm:main}, by showing that a certain function must be constant, in order for the bisectors to be equal.

Throughout the paper, we will be dealing with generic bisectors.

\section{Preliminaries}\label{sec:prelim}
We denote the bidisk as follows\,:
\begin{equation*}
\B = \HP \times \HP .
\end{equation*}
Points in the bidisk are pairs $\vx = (x_1, x_2)$, where
$x_i$ are points in the hyperbolic plane $\HP$.

In this paper, we consider the usual $L^2$ metric, so that the distance between two points $\vx,\vy\in\B$ is\,:
\[
\BHD(\vx, \vy) = \sqrt{ \HDS(x_1, y_1) + \HDS (x_2, y_2) }
\]
 where $d( \cdot , \cdot )$ denotes the distance in $\HP$.
The projection of geodesics in $\B$ are geodesics (or points) in $\HP$.

{\bf Convention for notation.}  For the remainder of the paper, we will use a letter in sans-serif font to denote a point in $\B$ and the same letter, in regular font, to denote its coordinates in $\HP$\,:
\begin{align*}
\vx & = (x_1,x_2) \\
\vy & = (y_1,y_2)
\end{align*}
etc.

\subsection{Bisectors}
Given $\vx,\vy\in\B$, their {\em bisector}, denoted $\E{\vx,\vy}$,
is the set of points equidistant to $\vx$ and $\vy$\,:
\begin{equation*}
\E{\vx,\vy}=\{\vp\in\B~\mid~\rho (\vx , \vp)=\rho (\vy , \vp)\} .
\end{equation*}
Writing the points in coordinates, the equality appearing in the definition may be re-written as follows\,:
\[
\begin{array}{rcl}
\sqrt{ \HDS (x_1,p_1) + \HDS (x_2, p_2) }  & = &\sqrt{ \HDS (y_1,p_1) +
\HDS (y_2, p_2) } \\
 \\
 \HDS (x_1,p_1) - \HDS (y_1, p_1)  & = & \HDS (y_2,p_2) -\HDS (x_2, p_2).
\end{array}
\]
Note that if $x_1=y_1$ or $x_2=y_2$, we are in the non-generic case where $\E{\vx,\vy}$ is the product of the hyperbolic plane with a geodesic in the hyperbolic plane.

The last equation motivates the following definition for the hyperbolic plane~\cite{CDD}. 
\begin{defn}
Let $x, y$ be distinct points in $\HP$ and let $k\in\R$. The \emph{square hyperbola of level $k$} is\,:
\[
\SH{k}{x,y} = \{ p\in \HP ~\vert~ \HDS( x,p) - \HDS (y,p) = k \}.
\]
Additionally, the \emph{level function} $L_{x,y}: \HP \rightarrow \R$ is\,:
\[
L_{x,y}(p) =  \HDS (x,p) - \HDS(y,p).
\]
\end{defn}
The level function is defined so that $ p \in \SH{L_{x,y} (p)}{ x, y}$.

Let $x,y\in\HP$ be distinct points.  Observe that any point $p\in \HP$ lies on the square
hyperbola of level
$L_{x,y} (p)$.

Denote by $\lambda$ the geodesic between $x$ and $y$.  For every $k\in\R$, the square hyperbola $\SH{k}{x, y}$ is symmetric with respect to
$\lambda$.  (See Lemma~\ref{lem:hypsymm}.)  When $k=0$, $\SH{0}{x, y}$ is itself a geodesic and it
intersects $\lambda$ at the midpoint between $x$ and $y$.

Reflection
in $\SH{0}{x, y}$
 interchanges $\SH{k}{x, y}$ with $\SH{-k}{x, y}$, in particular $x$ and $y$,  while
preserving (as a set) the geodesic $\lambda$.

A generic bisector admits a natural foliation
by products of square hyperbolae\,:
\begin{equation}\label{eq:foliation}
\E{\vx, \vy}= \bigcup_{k\in\R}   \SH{k}{x_1, y_1}\times
\SH{-k}{x_2, y_2}.
\end{equation}

Given $\vx,\vy\in \B$ and any  $p\in\HP$, there
exists a point $\vp\in\B$ lying on the bisector $\E{\vx, \vy}$ with $p$ as
one of its coordinates. For example, given $p_1\in\HP$, choose any
$p_2\in\SH{- L_{x_1, y_1} (p_1 )}{x_2, y_2}$.

The leaf $\SH{0}{x_1, y_1}\times \SH{0}{x_2, y_2}$ is called the \emph{spine} of the
bisector. The spine of a bisector is a \emph{flat}\,: it is a geodesically complete surface with zero
sectional curvature. Moreover, spines are the only square hyperbolae which are flats.

\section{Uniqueness of leaves}\label{sec:leaves}
In this section, we will see that if two generic bisectors are equal, then they must share the same leaves, possibly for different values of the level function.

\begin{lemma}\label{lem:samehyperbolae}
Let $\vx,\vy,\vu,\vv\in\B$ such that $\E{\vx ,\vy} = \E{\vu, \vv}$.  Assume that the bisector is generic. Then for every $k\in\R$, there exists $m\in\R$ such that\,:
\begin{equation*}
\SH{k}{x_1,y_1}\times\SH{-k}{x_2,y_2}=\SH{m}{u_1,v_1}\times\SH{-m}{u_2,v_2}.
\end{equation*}
Furthermore, if $k=0$, then $m=0$.
\end{lemma}
\begin{proof}
The last statement follows immediately from the observation that the spine is the unique flat in a bisector.

Let $\vq\in\E{\vx ,\vy} = \E{\vu, \vv}$.  Let $k,m\in\R$ such that\,:
\begin{align*}
(q_1 , q_2) & \in \SH{k}{x_1,y_1} \times \SH{-k}{x_2,y_2} \\
(q_1 , q_2) & \in \SH{m}{u_1,v_1} \times \SH{-m}{u_2,v_2}.
\end{align*}
Then\,:
\begin{align*}
L_{x_1, y_1} ( q_1) & = k \\
L_{u_1, v_1} ( q_1) & =m.
\end{align*}

By definition of the bisector, for any $p_2\in\SH{-k}{x_2,y_2}$, the point $(q_1,p_2)$ belongs to $\E{\vx ,\vy} $.
Since the bisectors are equal, $(q_1,p_2)$ also belongs to $\E{\vu ,\vv} $ and thus\,:
\begin{equation*}
L_{u_2, v_2} ( p_2)  =-m.
\end{equation*}
Since $p_2$ was chosen arbitrarily on $\SH{-k}{x_2,y_2}$ (and applying the same argument to $\SH{-m}{u_2,v_2}$), it follows that\,:
\begin{equation*}
\{q_1\}\times\SH{-k}{x_2,y_2}=\{q_1\}\times\SH{-m}{u_2,v_2}.
\end{equation*}
Repeating the same argument in the first factor, we have\,:
\begin{equation*}
\SH{k}{x_1,y_1} \times\SH{-k}{x_2,y_2}=\SH{m}{u_1,v_1} \times\SH{-m}{u_2,v_2}
\end{equation*}
as desired.
\end{proof}

Consequently, when two generic bisectors are equal, the square hyperbolae must be the same in each factor, up to reparametrizing the level function.
This motivates the following definition.

\begin{defn}
Let $\{x,y\},\{u,v\}\subset\HP$ be two pairs of distinct points.  We say that $\{x,y\}$ and $\{u,v\}$ are {\em SH-related} if there exists a function $m:\R\rightarrow\R$, with $m(0)=0$, such that,
for every $k\in\R$\,:
\begin{equation*}
\SH{k}{x,y}=\SH{m(k)}{u,v}.
\end{equation*}
\end{defn}

In particular, in the generic case, $\E{\vx ,\vy}=\E{\vu,\vv}$ if and only if the following two conditions hold\,:
\begin{itemize}
\item for $i=1,2$, $\{x_i,y_i\}$ and $\{u_i,v_i\}$ are SH-related;
\item if $m$ is the reparametrization in one factor of $\B$, then the reparametrization in the other factor is $k\mapsto -m(-k)$.
\end{itemize}

Observe that if $\{x,y\}$ and $\{u,v\}$ are SH-related and, say, $x=u$, then necessarily, $y=v$. Therefore, SH-related pairs are either equal or disjoint.

\section{Collinearity}\label{sec:collinear}
In this section, we will show that SH-related pairs of points must be collinear.  Our strategy will involve the closest point of a square hyperbola
to the common bisector.

Explicitly, let $x,y\in\HP$; we will show that $y$ is the point belonging to $\SH{L_{x,y}(y)}{x,y}$ which is closest to $\SH{0}{x,y}$.
(Experimentally, we have observed that indeed, for all real $k\neq 0$, the intersection of $\SH{k}{x,y}$ with the geodesic containing $x,y$ is the point which is closest to
$\SH{0}{x,y}$; however it is easier to prove in the particular case where that intersection point is either $x$ or $y$.)

One step in the argument will require a fact mentioned in~\S\ref{sec:prelim}, namely that the square hyperbola $\SH{k}{x,y}$ is symmetric with respect to the geodesic containing $x,y$.

\begin{lemma}\label{lem:hypsymm}%
Let $x,y$ be distinct points in $\HP$ and let $\lambda$ be the geodesic containing $x,y$. Then for every $k\in\R$, the square hyperbola
$\SH{k}{x , y} $ is symmetric with respect to $\lambda$.
\end{lemma}

\begin{proof}
Let $R$ denote the reflection in the line $\lambda$; then $R$ is an
isometry which fixes both $x$ and $y$.  Therefore\,:
$$
\begin{array}{rcl}
 L_{x,y}(R(p)) & =  & \HD^2(x,R(p)) - \HD^2(y, R(p)) \\
  & = &  \HD^2(R(x),R^2(p)) - \HD^2(R(y), R^2(p))\\
 & = &  \HD^2(x,p) - \HD^2(y,p)= L_{x,y}(p)
\end{array}
$$
Thus $R$ maps the square hyperbola $\SH{k}{x,y}$ to itself.
\end{proof}

Our argument centers around a certain \emph{Lambert quadrilateral}.  A Lambert quadrilateral in $\HP$ is a
quadrilateral with three right angles. The angle at the fourth vertex is necessarily
acute.
\begin{lemma}\label{lem:Lambert}
Let $Q$ be a Lambert quadrilateral whose sides meeting at the acute angle have lengths $\da$ and $\dk$, and whose diagonal meeting the acute angle
has length $\db$.  Then\,:
\begin{equation}\label{eq:LambertSinh}
\sinh^2 \da +\sinh^2 \dk  =
\sinh^2 \db
\end{equation}
and
\begin{equation}\label{eq:LambertLengths}
\da^2 + \dk^2 > \db^2 .
\end{equation}
\end{lemma}
\begin{proof}
Let $a,k$ respectively denote the sides of length $\da,~\dk$.  The diagonal, denoted $b$ in Figure~\ref{fig:Lambert}, divides $Q$ into two right triangles\,: $\Delta_a$,
containing $a$, and $\Delta_k$, containing $k$.  Let $\alpha$ be the angle at the vertex of $\Delta_a$, opposite to $a$.  Then the angle of the vertex of $\Delta_k$,
opposite to $k$, is $\pi/2-\alpha$.  (See Figure~\ref{fig:Lambert}.)
\begin{figure}
\begin{center}
   \ps{a}[][][1.0]{$a$}
   \ps{b}[][][1.0]{$b$}
   \ps{k}[][][1.0]{$k$}
   \ps{aa}[][][1.0]{$\alpha$}
   \ps{pa}[][][1.0]{$\frac{\pi}{2}-\alpha $}
\includegraphics[height=4cm]{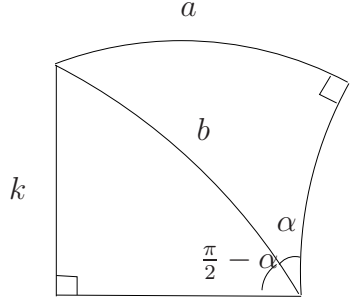}
\caption{A Lambert quadrilateral. The angle at the vertex adjacent to $a$ and $k$ is acute.  The diagonal is denoted $b$. }
\label{fig:Lambert}
\end{center}
\end{figure}

Applying the hyperbolic law of sines to $\Delta_a$ and $\Delta_k$\,:

\begin{align*}
\sinh \db & = \frac{\sinh \dk }{\sin (\alpha )}  \\
\sinh \db & = \frac{\sinh \da }{\sin (\pi/2-\alpha )}.
\end{align*}

Therefore\,:
\begin{align*}
\sinh^2 \dk + \sinh^2 \da  &= \left( \sin^2 (\alpha ) + \sin^2 (\pi/2-\alpha ) \right)  \sinh^2 \db  \\
 &=\sinh^2 \db .
\end{align*}
This proves \eqref{eq:LambertSinh}.

To prove \eqref{eq:LambertLengths}, note that
the Maclaurin Series for $\sinh^2 x$ is\,:
\[
\sinh^2 x = \sum_{m =1}^{\infty} c_{2m} x^{2m}
\]
where $c_{2m}=\frac{2^{2m-2}}{(2m)!} >0$.

Substituting \eqref{eq:LambertSinh} into the Maclaurin Series above, we get\,:
\[
\sum_{m =1}^{\infty} c_{2m} \left( \db^{2m} - \dk^{2m} - \da^{2m} \right) = 0.
\]

Now, suppose that $\db^2 \geq \da^2 + \dk^2$. Raising both sides to the power $m >1$\,:
\[
\db^{2m} \geq \left(  \da^2 + \dk^2 \right)^m >  \da^{2m} + \dk^{2m}.
 \]
The second inequality is strict, implying\,:
\[
\sum_{m =1}^{\infty} c_{2m} \left( \db^{2m} - \dk^{2m} - \da^{2m} \right) > 0
\]
which is a contradiction.
Therefore, \eqref{eq:LambertLengths} holds.
\end{proof}

\begin{lemma}\label{lem:closestpoint}%
Let $x,y$ be distinct points in $\HP$. Then $y$ is the point on the square hyperbola $\SH{L_{x,y}(y)}{x,y}$ which is closest to the bisector
$\SH{0}{x , y}$.
\end{lemma}

\begin{proof}
Let $\dk=d(x,y)/2$, so that $\dk$ is the distance between $y$ and the bisector $\SH{0}{x , y}$.  Let $g_t$ be the one-parameter subgroup of
isometries whose invariant line is $\SH{0}{x,y}$.  Since the curve $g_t(y)$ is equidistant
from the bisector,  it suffices to show that for every real $t\neq 0$\,:
\begin{equation*}
L_{x,y}(g_t(y))<4\dk^2.
\end{equation*}
Indeed, this will imply that for every point of $\SH{L_{x,y}(y)}{x,y}$ other than $y$, the distance from the bisector is greater than $4\dk^2$.

Denote by $\lambda$ the geodesic containing $x,y$.  Let $w=\lambda\cap\SH{0}{x,y}$.  Thus $w$ is the midpoint between $x$ and $y$ and\,:
\begin{equation*}
d(y,w)=\dk.
\end{equation*}

Let $t\neq 0$ be an arbitrary real number.  Set $\da$ to be the number such that\,:
\begin{equation*}
d(y,g_{2t}(y))=2\da.
\end{equation*}

Let $\mu$ be the geodesic containing $g_t(w)$ which is perpendicular to $\SH{0}{x,y}$.  Since reflection in $\mu$ leaves this bisector invariant, and maps $y$ to $g_{2t}(y)$ (and vice versa),
it also leaves invariant the geodesic segment between $y$ and $g_{2t}(y)$.  Set $z$ to be the intersection of this geodesic segment and $\mu$.  Then,
$y,w,g_t(w),z$ form a Lambert quadrilateral, with acute angle at the vertex $y$, and\,:
\begin{equation*}
d(y,z)=\da.
\end{equation*}
See Figure~\ref{fig:close1}.
\begin{figure}
\begin{center}
   \ps{y}[][][0.8]{$y$}
   \ps{K}[][][0.8]{$K$}
   \ps{z}[][][0.8]{$w$}
   \ps{x}[][][0.8]{$x$}
   \ps{A}[][][0.8]{$A$}
   \ps{B}[][][0.8]{$ $}
   \ps{s}[][][0.8]{$g_t(w)$}
   \ps{t}[][][0.8]{$z$}
   \ps{u}[][][0.8]{$ $}
   \ps{q}[][][0.8]{$ $}
   \ps{r}[][][0.8]{$g_{2t}(w)$}
   \ps{p}[][][0.8]{$g_{2t}(y)$}
   \ps{m}[][][0.8]{$\mu$}
   \ps{l}[][][0.8]{$SH_0(x,y)$}
\includegraphics[height=8cm]{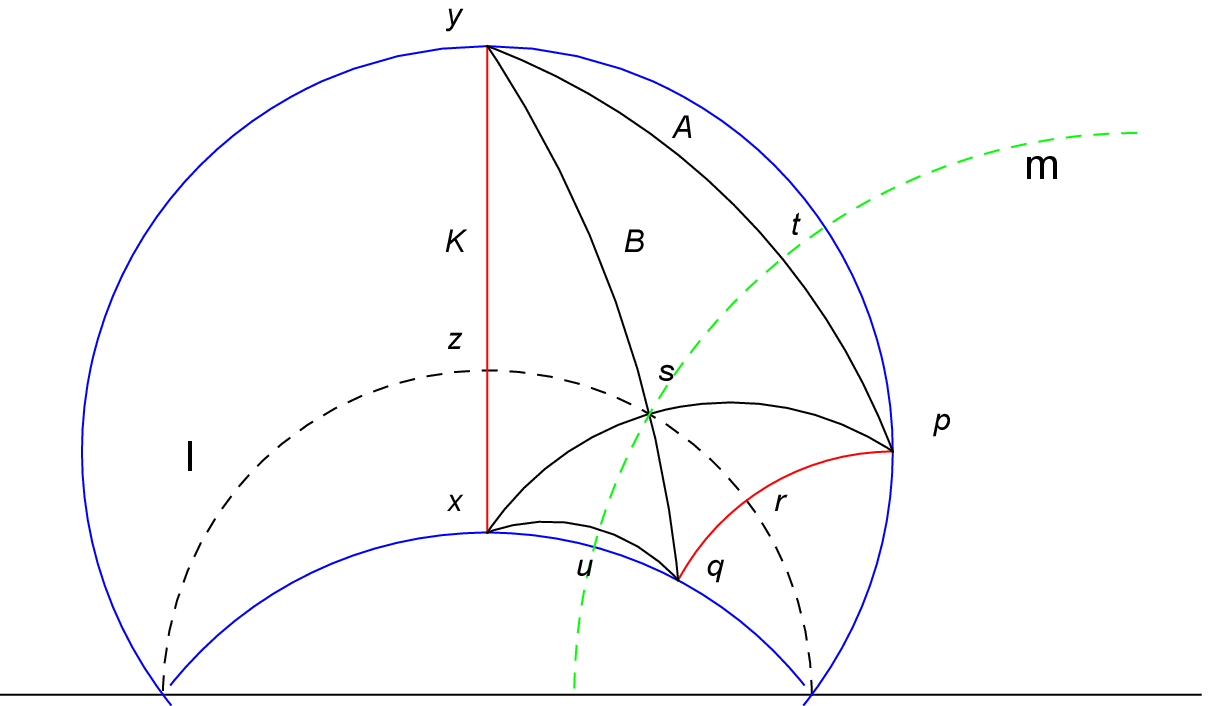}
\caption{The point $z$ is the midpoint between $y$ and $g_{2t}(y)$, by construction. }
\label{fig:close1}
\end{center}
\end{figure}

Finally, set $\db$ to be the number such that\,:
\begin{equation*}
d(x,g_{2t}(y))=2\db.
\end{equation*}
We claim that\,:
\begin{equation*}
d(y,g_{t}(w))=\db.
\end{equation*}
If the claim is true, then applying Lemma~\ref{lem:Lambert} to the Lambert quadrilateral above, we obtain\,:
\begin{equation*}
\dk^2>\db^2-\da^2
\end{equation*}
yielding the Lemma.  To prove the claim, observe that\,:
\begin{itemize}
\item reflection in $\mu$ maps the line segment between $g_t(w)$ and $g_{2t}(y)$ to the line segment between $g_t(w)$ and $y$;
\item reflection in $\SH{0}{x,y}$, composed with reflection in $\mu$, maps the line segment between $g_t(w)$ and $g_{2t}(y)$ and
the line segment between $g_t(w)$ and $x$;
\item $x, g_t(w), g_{2t}(y)$ must be collinear, since isometries preserve angles.
\end{itemize}
Thus $d(y,g_{t}(w))=d(g_{2t}(y),x)/2=\db$ as claimed.
\end{proof}

Finally, we are ready to prove that SH-related pairs must be collinear.
\begin{prop}
Let $x,y,u,v\in\HP$ be distinct points such that $\{x,y\}$ and $\{u, v\}$ are SH-related.  Then $x,y,u,v$ are collinear.
\end{prop}

\begin{proof}
Let $\lambda$ be the geodesic containing and $x,y$ and set $\dk = d(x,y)/2$. Recall that $\SH{0}{x,y} = \SH{0}{u,v}$.
By  Lemma~\ref{lem:closestpoint}, the point $y$ is the
unique point on the square hyperbola
$\SH{4\dk^2}{x,y}$ that is closest to $\SH{0}{x,y}$.

By hypothesis, there exists $m\in\R$ such that $\SH{m}{u,v} =\SH{4\dk^2}{x,y}$.

Now let $\mu$ be the geodesic containing $u,v$.  By Lemma~\ref{lem:hypsymm}, the square hyperbola $\SH{m}{u,v}$ is
symmetric with respect to $\mu$.  Therefore, except for the point at the intersection of $\mu$ and $\SH{m}{u,v}$, there are always two points on $\SH{m}{u,v}$ which are at equal distance
from $\SH{0}{u,v}$.

Since $y$ is the unique point on $\SH{m}{u,v}$ at distance $\dk$ from $\SH{0}{u,v}$, it follows that $y\in\mu$ and thus $\mu=\lambda$.
\end{proof}

\section{Growth of quotients of level functions for square hyperbolae}\label{sec:SH}

We have shown that, in order for two bisectors to be equal, their leaves must be the same and the coordinates of the points must be collinear in each factor. To complete the proof of Theorem~\ref{thm:main}, we undertake a careful examination of the square hyperbolae in one factor of the bidisk, assuming collinearity of the coordinates.
We will show that if collinear pairs of points are distinct, then they cannot be SH-related. Thus, by Lemma~\ref{lem:samehyperbolae}
the bisectors of different pairs of points must be different.

%
%
 %
%

Given an ordered set $\Omega=\{x, y, u, v\}\subset\HP$, the following function is well-defined\,:
\begin{equation}\label{eq:quotient}
\begin{split}
\Phi_\Omega:\HP\setminus\SH{0}{u,v} & \longrightarrow \R \\
 p & \longmapsto \frac{L_{x,y}(p)}{L_{u,v}(p)}.
\end{split}
\end{equation}

\begin{lemma}
 Let $\Omega=\{x, y, u, v\}\subset\HP$ be an ordered set of distinct points which are collinear. Assume furthermore that $\{x,y\}$ and $\{u,v\}$ are SH-related.
 Then $\Phi_\Omega$ is constant.
\end{lemma}
\begin{proof}
Let $\lambda$ be the line containing $ x, y, u, v$, and $w$, the midpoint between $x$ and $y$ (or $u$ and $v$).
Relabeling the points if necessary, we
may assume that $x$ and $u$ lie on the same side of $\SH{0}{x,y} =\SH{0}{u,v}$.

Consider any point $p \in\lambda$ which does not belong to $\SH{0}{x,y}$, and on the side
opposite of $\SH{0}{x,y}$ from $x$ and $u$.
Set $\da = \HD(x,w) = \HD(y,w)$, $\db= \HD(u,w)=\HD(v,w)$ and $\dc=\HD(p,w)$.
The following holds:
\[
 \frac{L_{ x,y }(p)}{L_{ u,v }(p)} = \frac{(\da+\dc)^2 - (\da-\dc)^2}{(\dc+\db)^2 - (\dc-\db)^2}
 = \frac{4\da\dc}{4\dc\db}=\frac{\da}{\db}.
\]
We obtain the same value for $p$ on the same side of $\SH{0}{x,y}$
as $x$ and $u$.  Thus the value of the function $\Phi_\Omega$ is constant for all
$p\in \lambda \setminus \SH{0}{x,y}$.

Consider a point $w\not\in \SH{0}{u,v}$, and let $k=L_{x,y}(w)$.  Observe that $L_{x,y}(w)=L_{x,y}(p)$, where
$p=\SH{k}{x,y}\cap\lambda$.  Now since $\{x,y\}$ and $\{u,v\}$ are SH-related, we have\,:
\begin{equation*}
\SH{k}{x,y}=\SH{m}{u,v}
\end{equation*}
for some $m\in\R$.  Thus $L_{u,v}(w)=L_{u,v}(p)$ as well, and\,:
\begin{align*}
\Phi_\Omega(w) & = \frac{L_{x,y}(w)}{L_{u,v}(w)} \\
& = \frac{L_{x,y}(p)}{L_{u,v}(p)} \\
& = \Phi_\Omega(p).
\end{align*}
Therefore, the value of $\Phi_\Omega$ is constant.
\end{proof}

For the remainder of this section, assume that $\Omega=\{x,y,u,v\}\subset\HP$ is an ordered set of collinear points, such that the pairs $\{x,y\}$ and $\{u,v\}$ share
 a common bisector,
yet such that $\{x,y\}\neq\{u,v\}$.  We will show that in this case, $\Phi_\Omega$ fails to be constant.  Consequently, by the preceding lemma, $\{x,y\}$
and $\{u,v\}$ are not SH-related.

Recall the following standard fact.
\begin{prop}[Hyperbolic Pythagorean Theorem]
Given a  right triangle in $\HP$ with
side lengths $\da , \db , \dc$, we have\,:
\[
\cosh \da \cosh \db = \cosh \dc .
\]
\end{prop}

Let $p\in\HP\setminus\SH{0}{x,y}$, and  let $\mu$
be the geodesic containing $p$
which is perpendicular to $\lambda$. Let $w= \lambda \cap \mu$.
 There are four right triangles with one common side $\overline{wp}$ and
another side lying along
$\lambda$ with vertices $x,y,u$ and $v$, respectively.  Set
$\da = \rho (w,x)$, $\db= \rho (w,y)$, $\dc = \rho (w,u)$,
$\dd= \rho (w,v)$ and $\dk= \rho(p,w)$.

\begin{figure}
\begin{center}
   \ps{v}[][][0.8]{$v$}
   \ps{w}[][][0.8]{$w$}
   \ps{y}[][][0.8]{$y$}
   \ps{p}[][][0.8]{$p_K$}
   \ps{l}[][][0.8]{$\lambda$}
   \ps{m}[][][0.8]{$\mu$}
   \ps{u}[][][0.8]{$u$}
   \ps{x}[][][0.8]{$x$}
\includegraphics[height=6cm]{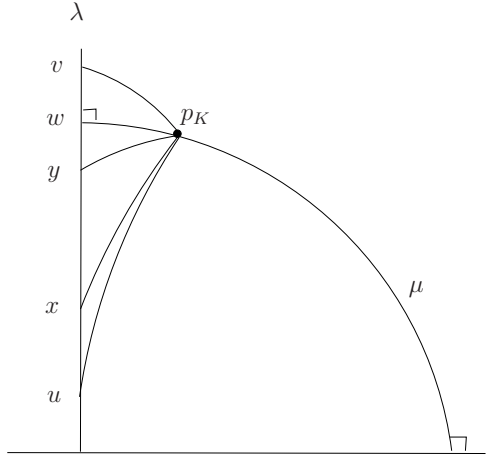}
\caption{Two pairs of collinear points. The point $p_K$ lies on a geodesic which is
perpendicular to $\lambda$, at a point $w$ which is at least as close to $v$ as it is to $y$.}
\label{fig:Expansion}
\end{center}
\end{figure}

The value of the level function is\,:
\[
L_{(x,y)}(p) = \left(  \arccosh  \left(\cosh \rho(p,x) \right) \right)^2 -
\left( \arccosh \left( \cosh \rho(p,y) \right) \right)^2
\]
and can be rewritten, using the Hyperbolic Pythagorean Theorem as\,:
\begin{equation}\label{eq:PythLevels}
L_{(x,y)}(p)  = \left(  \arccosh  \left(\cosh \da \cosh \dk \right) \right)^2 -
\left( \arccosh \left( \cosh \db \cosh \dk \right) \right)^2.
\end{equation}

Consider each of the two terms in \eqref{eq:PythLevels}\,:
\[
f(s,t) = \left( \arccosh ( \cosh s \cosh t ) \right)^2
\]
and expand $f(s , t )$  in the variable $s$ around $s=0$ to get\,:
\begin{equation}\label{eq:Taylor}
 f(s , t ) = t^2 + \left(t \coth t \right) s^2 + O( s^4).
\end{equation}

\begin{prop}
 Let $\Omega=\{x,y,u,v\}\subset\HP$ be an ordered set of distinct, collinear points, such that $\SH{0}{x,y}=\SH{0}{u,v}$.  Then $\Phi_\Omega$ is not constant.
\end{prop}
\begin{proof}
Relabeling if necessary, we may assume that the points are placed on $\lambda$ as in Figure~\ref{fig:Expansion} \,: $v,y,x,u$.
Fix a point $w\in \lambda$ such that
\[
\HD(x,w) > \HD(y,w)\geq \HD(v,w).
\]
Using values defined above,  we have  $\dc>\da>\db\geq \dd $.

Define $\mu$  to be the geodesic perpendicular to $\lambda$ at the
fixed point $w$, and let $p_{\dk}$ be the
set of points on $\mu$ on one side of $\lambda$ parameterized by the
distance $\dk = \HD(p_{\dk},w)$.
The expression
\eqref{eq:quotient} is now written
\begin{align*}
\frac{L_{x,y}(p_{\dk})}{ L_{u,v}(p_{\dk})} & =
\frac{\left( \da^2 + (\da \coth \da) \dk^2 + O(\dk^4) \right)-
\left( \db^2 + (\db \coth \db) \dk^2 + O(\dk^4) \right)}
{\left( \dc^2 + (\dc \coth \dc) \dk^2 + O(\dk^4) \right)-
\left( \dd^2 + (\db \coth \dc) \dk^2 + O(\dk^4) \right)}
\\
& =
\frac{\left( \da^2 - \db^2 \right) +  \left( \da \coth \da -
      \db \coth \db \right) \dk^2 + O(\dk^4)}
{\left( \dc^2 - \dd^2 \right) + \left(\dc \coth \dc -
         \dd \coth \dd \right) \dk^2 + O(\dk^4)}.
\end{align*}

The function $g(x) = x \coth x $ is increasing for $x>0$, so that
\[
\left(\dc \tanh \dc - \dd \tanh \dd \right)  > \left( \da \tanh \da - \db \tanh \db \right).
\]
because of the inequalities established above. Therefore, the function
$\Phi_\Omega(p_{\dk})$ must be decreasing in $\dk$ for small
values of $\dk\geq 0$.
\end{proof}


\end{document}